\documentclass[11pt]{amsart}

\usepackage{lscape}

\newcounter{cnt1}
\newcounter{cnt2}
\newcommand{\blr}{\begin{list}{$(\roman{cnt1})$}
    {\usecounter{cnt1} \setlength{\topsep}{0pt}
        \setlength{\itemsep}{0pt}}}
\newcommand{\bla}{\begin{list}{$($\alph{cnt2}$)$}
    {\usecounter{cnt2} \setlength{\topsep}{0pt}
        \setlength{\itemsep}{0pt}}}
\newcommand{\el}{\end{list}}
\newtheorem{thm}{Theorem}
\newtheorem{lem}[thm]{Lemma}

{\newtheorem{Def}[thm]{Definition}
\newtheorem{prop}[thm]{Proposition}
\newtheorem{rem}[thm]{Remark}
\newcommand{\Rem}{\begin{rem} \rm}
\newcommand{\bdfn}{\begin{Def} \rm}
\newcommand{\edfn}{\end{Def}}

\begin{document}
\large
\title[Biadjoint operators]{ Orthogonality for biadjoints of operators }
\author[Rao]{T. S. S. R. K. Rao}
\address[T. S. S. R. K. Rao]
{Department of Mathematics\\ Shiv Nadar University \\\ Delhi-NCR\\Adjunct Professor\\CARAMS\\
Manipal Academy of Higher Education \\
Manipal\\ India,
\textit{E-mail~:}
\textit{srin@fulbrightmail.org}}
  \maketitle
\begin{abstract}
Let $X,Y$ be non-reflexive Banach spaces. Let ${\mathcal L}(X,Y)$ be the space of bounded operators from $X$ to $Y$. For $T \in {\mathcal L}(X,Y)$, and a closed subspace $Z \subset Y$, this paper deals with the question, if $T \perp {\mathcal L}(X,Z)$ in the sense of Birkhoff-James, when is $ T^{\ast\ast} \perp {\mathcal L}(X^{\ast\ast},Z^{\bot\bot}) $? If $Z \subset Y$ is a subspace of finite codimension which is the kernel of projection of norm one, we show that this is always the case. Moreover in this case, there is an extreme point $\Lambda $ of the unit ball of the bidual $ X^{\ast\ast}$  such that $\|T^{\ast\ast}\|= \|T^{\ast\ast}(\Lambda)\|$ and $T^{\ast\ast}(\Lambda) \perp Z^{\bot\bot}$.

\end{abstract}
\section { Introduction}
Let $X,Y$ be, real, non-reflexive Banach spaces. An important problem in operator theory is to consider standard Banach space theoretic operations on the space of operators, ${\mathcal L}(X,Y)$ or its subspaces, like the space of compact operators, weakly compact operators, etc., and see if the resulting structure is again (or in ) a space of operators, perhaps between different Banach spaces. For example, one important property that has attracted a lot of attention recently (and also has some bearing on the current investigation), is the quotient space problem. For Banach spaces $X,Y,Z$ with $Z \subset Y$, in the space of operators ${\mathcal L}(X,Y)$, when one performs the `quotient operation' with the subspace ${\mathcal L}(X,Z)$, a natural question is to ask, when is the quotient space ${\mathcal L}(X,Y)/{\mathcal L}(X,Z)$, surjective isometric to a space of operators? Here a natural target space is ${\mathcal L}(X,Y/Z)$.
If $\pi: Y \rightarrow Y/Z$ denotes the quotient map, then a possible isometry between the spaces involved is, $[T] \rightarrow \pi \circ T$ where the equivalence class, $[T] \in {\mathcal L}(X,Y)/{\mathcal L}(X,Z)$.
See \cite{MBF}, \cite{BFR} and \cite{R2} .

\vskip 1em

 In this paper we investigate similar aspects for the Banach space theoretic property of Birkhoff-James orthogonality and the operator theoretic property of taking biadjoints of an operator. 
Let $Z \subset Y$ be a closed subspace, let $y \in Y$. We recall that $y \perp Z$ in the sense of Birkhoff-James, if $\|y+z\|\geq \|y\|$ all $z \in Z$. Deviating from standard notation, we denote the set of all such $y$ by $Z^{\dagger}$. For a closed subspace $M \subset {\mathcal L}(X,Y)$, with $T \perp M$,  we are interested in the question, when is $T^{\ast\ast}$ orthogonal to the smallest weak$^\ast$-closed subspace of ${\mathcal L}(X^{\ast\ast},Y^{\ast\ast})$, containing $\{S^{\ast\ast} : S \in M\}$. 
\vskip 1em

 Let $Z \subset Y$ be a closed, non-reflexive subspace. We consider the case $M = {\mathcal L}(X,Z)$. We use the canonical embedding of a Banach space $X$ in its bidual $X^{\ast\ast}$ and also recall the identification of the bidual of the quotient space $(Y/Z)^{\ast\ast}$ as $ Y^{\ast\ast}/Z^{\bot\bot}$ and $Z^{\ast\ast}$ as $Z^{\bot\bot} \subset Y^{\ast\ast}$. Using these embedding, we have, $d(y,Z) = d(y,Z^{\bot\bot})$ for all $y \in Y$. For $T \in {\mathcal L}(X,Y)$, the biadjoint, $T^{\ast\ast} \in {\mathcal L}(X^{\ast\ast},Y^{\ast\ast})$ and when $T$ is $Z$-valued, $T^{\ast\ast}$ takes values in $Z^{\perp\perp}$ (we also note for a future use, when $T$ is compact, $T^{\ast\ast}$ takes values in $Z$).
 \vskip 1em

 So in this case our problem can be formulated as, if $T \perp M $, when is $T^{\ast\ast} \perp {\mathcal L}(X^{\ast\ast},Z^{\bot\bot})$? We note that
 ${\mathcal L}(X^{\ast\ast},Z^{\bot\bot})$ is closed in the weak$^\ast$-operator topology of ${\mathcal L}(X^{\ast\ast},Y^{\ast\ast})$, and
 we show that this is the smallest closed subspace containing $M$. So that by solving this orthogonality problem, we have a solution to the biadjoint question proposed at the beginning, somewhat analogous to the von Neumann double commutant theorem for $C^\ast$-algebras (interpreted as concrete identification of the closure of a space, in the weak-operator topology) . We refer to
 \cite{DU} for standard results on spaces of operators and tensor products. 
 \vskip 1em
 This paper is an expanded version of my talk given at ICLAA2021,  organised by CARAMS, MAHE, Manipal during December 15-17, 2021. I thank the Editors of the ICLAA-special volume for their efficient handling of the MS during the pandemic. Thanks are also due to the referee for the extensive comments.
 \section{Main Result}
 
 \vskip 1em
 
 For a Banach space $X$, we denote the closed unit ball  by $X_1$ . Following the lead of the well-known Bhatia-Semrl theorem \cite{BS} on point-wise orthogonality
 of two orthogonal operators (established in the case of finite dimensional Hilbert spaces),
 one can ask, if $T \perp {\mathcal L}(X,Z)$, does there exist an extreme point $\Lambda \in X^{\ast\ast}_1$ such that $\|T\|=\|T^{\ast\ast}\|= \|T^{\ast\ast}(\Lambda)\|$ and $T^{\ast\ast}(\Lambda) \perp Z^{\bot\bot}$?
 \vskip 1em
 
 Most generalizations of the Bhatia-Semrl theorem for Hilbert spaces or for operators between Banach spaces (see the survey article \cite{PS}) require
 $T$ to attain its norm, which is a strong requirement even when $X$ is reflexive. On the other hand, it is well-known that if $T$ is a compact operator, $T^{\ast\ast}$ attains its norm at an extreme point and also a result of V. Zizler, \cite{Z},  asserts that $\{T \in {\mathcal L}(X,Y): T^{\ast\ast}~attains~its~norm\}$ is norm dense in ${\mathcal L}(X,Y)$. Thus anticipating $T^{\ast\ast}$ to attain its norm is a more natural condition for non-reflexive Banach spaces. Now we propose a general interpretation of the Bhatia-Semrl theorem as ` when does global behaviour  on an operator, leads to a similar  local behaviour at points where the biadjoint attains its norm?' Our generalization here involves orthogonality of an operator to a space of operators. See \cite{R1} for another interpretation.
 \vskip 1em
 We will first investigate the sufficient condition, leading to very interesting minimax formulae, see condition (3) in Proposition 1 below.
 It is well known and easy to see that $y \perp Z$ if and only if $d(y,Z) = \|y\|$.
 
 \begin{prop} Suppose $\|T\|=\|T^{\ast\ast}\|= \|T^{\ast\ast}(\Lambda)\|$ for some $\Lambda \in X^{\ast\ast}_1$ and $T^{\ast\ast}(\Lambda) \perp Z^{\bot\bot}$. The following hold.
 \begin{enumerate}

\item $T^{\ast\ast} \perp {\mathcal L}(X^{\ast\ast},Z^{\bot\bot})$ as 
 well as $T \perp {\mathcal L}(X,Z)$.
 \item Now using the distance interpretation of orthogonality
 for subspaces, in this case we have,$$ \|T\|=\|T^{\ast\ast}\|= d(T,{\mathcal L}(X,Z)) = d(T^{\ast\ast}, {\mathcal L}(X^{\ast\ast},Z^{\bot\bot})).$$
 \item $d(T^{\ast\ast}, {\mathcal L}(X^{\ast\ast},Z^{\bot\bot}))= sup_{\tau \in X^{\ast\ast}_1} \{d(T^{\ast\ast}(\tau),Z^{\bot\bot})\} = d(T^{\ast\ast}(\Lambda), Z^{\bot\bot})$.
\end{enumerate}
 	
 \end{prop}
\begin{proof}

 To see 1) and 2), for any $S \in {\mathcal L}(X^{\ast\ast},Z^{\bot\bot})$, $\|T^{\ast\ast}-S\|\geq \|T^{\ast\ast}(\Lambda)-S(\Lambda)\| \geq \|T^{\ast\ast}(\Lambda)\|= \|T\|$, since $S(\Lambda) \in Z^{\bot\bot}$. Thus $T^{\ast\ast} \perp {\mathcal L}(X^{\ast\ast},Z^{\bot\bot})$ as 
 well as $T \perp {\mathcal L}(X,Z)$. Now using the distance interpretation of orthogonality
 for subspaces, in this case we have,$$ \|T\|=\|T^{\ast\ast}\|= d(T,{\mathcal L}(X,Z)) = d(T^{\ast\ast}, {\mathcal L}(X^{\ast\ast},Z^{\bot\bot})).$$
 \vskip 1em
 To see 3), by the distance interpretation of orthogonality again, $$(*)~~ d(T^{\ast\ast}(\Lambda),Z^{\bot\bot}) = \|T^{\ast\ast}(\Lambda)\| \leq sup_{\tau \in X^{\ast\ast}_1} \{d(T^{\ast\ast}(\tau),Z^{\bot\bot})\}. $$
 
 Fix $S \in {\mathcal L}(X^{\ast\ast},Z^{\bot\bot})$ and $\tau \in X^{\ast\ast}_1$.
 $$d(T^{\ast\ast}(\tau),Z^{\bot\bot}) \leq \|T^{\ast\ast}(\tau)-S(\tau)\|\leq \|T^{\ast\ast}-S\|.$$
 Therefore $sup_{\tau \in X^{\ast\ast}_1} \{d(T^{\ast\ast}(\tau),Z^{\bot\bot})\} \leq d(T^{\ast\ast},Z^{\bot\bot})$.
 Thus by $(*)$ and (2), we get (3).

 \end{proof}
 \vskip 1em
\begin{rem}
	
We note that the conclusions of Proposition 1 are also valid if we replace bounded operators by compact operators. In Proposition 1, suppose $T$ is also a compact operator. Since $T^{\ast\ast}(\Lambda) \perp Z$, we also have, $$\|T^{\ast\ast}(\Lambda)\| = d(T^{\ast\ast}(\Lambda),Z).$$
 
 \end{rem}
 \begin{rem} It is of independent interest to note that, for $T \in {\mathcal L}(X,Y)$, if 
 $$d(T, {\mathcal L}(X,Z))= sup_{x \in X_1} \{d(T(x),Z)\}.$$ then again (3) holds. To see this note that $d(T(x),Z) = d(T^{\ast\ast}(x),Z^{\bot\bot})$. Now $$d(T^{\ast\ast},{\mathcal L}(X^{\ast\ast},Z^{\bot\bot}) \leq d(T,{\mathcal L}(X,Z))$$
 and it follows from the proof of (3),
 $$	sup_{x \in X_1} \{d(T(x),Z)\} \leq sup_{\tau \in X^{\ast\ast}_1} \{d(T^{\ast\ast}(\tau),Z^{\bot\bot})\} \leq d(T^{\ast\ast},Z^{\bot\bot}).$$
 Thus (3) holds.
 \end{rem}
 Our aim in this paper is to partially solve the converse problem, for all domains $X$ and imposing conditions only on the pair $(Y,Z)$. 
\vskip 1em
 Since Proposition 1 is a solution to an optimization problem
 (see \cite{R}) on the weak$^\ast$-compact convex set $X^{\ast\ast}_1$, we next show that if
 there is a solution, then there is an extremal solution. We recall that for a compact convex set $K$, $E \subset K$ is an extreme set, if for $x,y \in K$, $\lambda \in (0,1)$, $\lambda x + (1-\lambda)y \in E$ implies $x,y \in E$.
 An extreme convex set is called a face. For $x \in E$, if $face(x)$ denotes the smallest face of $K$, containing $x$, then clearly $E = \cup\{face(x): x \in E\}$. See \cite{A} page 122.
 We apply these ideas to $K = X^{\ast\ast}_1$, equipped with the weak$^\ast$-topology, along with a Krein-Milman type theorem to get the extremal solution.
 \begin{prop}
 Let $\Lambda \in X^{\ast\ast}_1$ be such that
 $$\|T^{\ast\ast}\|= \|T^{\ast\ast}(\Lambda)\|~and~	
 sup_{\tau \in X^{\ast\ast}_1} \{d(T^{\ast\ast}(\tau),Z^{\bot\bot})\} = d(T^{\ast\ast}(\Lambda), Z^{\bot\bot}).$$ Then there is an extreme point of $X^{\ast\ast}_1$, satisfying the same conditions.

 \end{prop}
\begin{proof}
Let $E = \{\tau_0 \in X^{\ast\ast}_1: \|T^{\ast\ast}(\tau_0)\|= \|T^{\ast\ast}\|~and~ sup_{\tau \in X^{\ast\ast}_1} \{d(T^{\ast\ast}(\tau),Z^{\bot\bot})\} = d(T^{\ast\ast}(\tau_o), Z^{\bot\bot})\}$.
It is easy to see that being an intersection of two extreme sets, $E$ is an extreme subset of $X^{\ast\ast}_1$.
Now we use a result of T. Johannesen (Theorem 5.8 in \cite{Lima}), which says that if the complement of $E$, $E^{c}$ in $X^{\ast\ast}_1$ is a union
of weak$^\ast$-compact convex sets, then $E$ has an extreme point of  $X^{\ast\ast}_1$ .
Note that $$E^{c} = \cup_n\{\tau' \in X^{\ast\ast}_1:
\|T^{\ast\ast}(\tau')\| \leq \|T^{\ast\ast}\|-\frac{1}{n}\}$$
$$\cup_n\{ \tau' \in X^{\ast\ast}_1: d(T^{\ast\ast}(\tau'),Z^{\bot\bot})\leq sup_{\tau \in X^{\ast\ast}_1} \{d(T^{\ast\ast}(\tau),Z^{\bot\bot})\}-\frac{1}{n}\}.$$
We next note that if $\pi: Y \rightarrow Y/Z$ is the quotient map, then $\pi^{\ast\ast}: Y^{\ast\ast} \rightarrow Y^{\ast\ast}/Z^{\bot\bot}$ is the canonical quotient map and $d(T^{\ast\ast}(\tau),Z^{\perp\perp}) = \|\pi^{\ast\ast}(\tau)\|.$
Thus by the convexity and weak$^\ast$-lower-semi-continuity of $\|.\|$ and the weak$^\ast$-continuity of the adjoint map, we see that all the sets in the union are weak$^\ast$-compact and convex. Thus $E$ has an extreme point of $X^{\ast\ast}_1$.
	
\end{proof}
We recall that ${\mathcal L}(X,Y^{\ast\ast})$ is the dual of the projective tensor product space $X \otimes_{\pi}Y^\ast$. See \cite{DU} Chapter VIII. Our next Lemma is folklore and is included here for the sake of completeness. We use the canonical identification of $Y^\ast/Z^\bot = Z^\ast$ and $\pi: Y^\ast \rightarrow Y^\ast/Z^\bot$ denotes the quotient map. For $\Lambda \in X^{\ast\ast}$ and $y^\ast \in Y^\ast$,  $\Lambda \otimes \pi(y^\ast)$ denotes the (weak$^\ast$-continuous functional) on ${\mathcal L}(X^{\ast\ast},Z^{\bot\bot})$, defined by $(\Lambda \otimes \pi(y^\ast))(S) = S(\Lambda)(y^\ast)$ for $S \in {\mathcal L}(X^{\ast\ast},Z^{\bot\bot})$. Thus 
$(\Lambda \otimes \pi(y^\ast))(T^{\ast\ast}) = \Lambda(T^\ast(\pi(y^\ast)))$, for any $T \in {\mathcal L}(X,Z)$.
\begin{lem}
	Let $Z \subset Y$ be a closed subspace and let $M = \{T^{\ast\ast}: T \in {\mathcal L}(X,Z)\}$. Then $M$ is weak$^\ast$-dense in ${\mathcal L}(X^{\ast\ast},Z^{\bot\bot})=(X^{\ast\ast} \otimes_{\pi} Y^\ast/Z^\bot)^\ast$.
\end{lem}
\begin{proof} Suppose for $\Lambda \in X^{\ast\ast}$ and $y^\ast \in Y^\ast$, $(\Lambda \otimes \pi(y^\ast)(T^{\ast\ast})=0$ for all $T \in {\mathcal L}(X,Z)$. For any $z \in Z$ choose $y^\ast$ such that $\pi(y^\ast)(z)=1$. For $x^\ast \in X^\ast$ and for the operator $T^\ast = z \otimes x^\ast$, we have $\Lambda(T^\ast(\pi(y^\ast))) = 0= \Lambda (x^\ast)$. Therefore $\Lambda =0$. Hence by an application of the separation theorem for the weak$^\ast$-topology, $M$ is weak$^\ast$-dense in ${\mathcal L}(X^{\ast\ast},Z^{\bot\bot})$.
	
\end{proof}
We recall that if $\pi: Y \rightarrow Y/Z$ is the quotient map, then $\pi^\ast$ is the inclusion map
of $Z^\bot$ in $Y^\ast$. In what follows, we use a
contractive projection to achieve the required compact `lifting'. See \cite{R2} for other alternate approaches.

 \begin{thm} Let $X$ be a Banach space and let $Z \subset Y$ be a closed finite codimensional subspace such that there is a contractive linear projection $P: Y \rightarrow Y$ such that $ker(P)=Z$. Let $T \in {\mathcal L}(X,Y)$ and $T \perp {\mathcal L}(X,Z)$. Then there is an extreme point $\Lambda \in X^{\ast\ast}_1$ such that $\|T\|=\|T^{\ast\ast}\|=\|T^{\ast\ast}(\Lambda)\|$ and $T^{\ast\ast}(\Lambda) \perp Z^{\bot\bot}$.
\end{thm}  
\begin{proof} Since $Z$ is of finite codimension, $\pi \circ T: X \rightarrow Y/Z$ is a compact operator. It is well-known (see \cite{HWW} page 226)
that there are extreme points $\Lambda \in X^{\ast\ast}_1$, $y^\ast \in Z^\bot_1$ such that
$\|\pi \circ T\|= \Lambda(T^\ast(y^\ast)) = T^{\ast\ast}(\Lambda)(y^\ast)\leq d(T^{\ast\ast}(\Lambda),Z^{\bot\bot}) \leq sup_{\tau \in X^{\ast\ast}_1}\{d(T^{\ast\ast}(\tau),Z^{\perp\perp})\}\leq d(T^{\ast\ast},{\mathcal L}(X^{\ast\ast},Z^{\bot\bot}))\leq d(T,{\mathcal L}(X,Z))= \|T\|. $

\vskip 1em
Now let $P': Y/Z \rightarrow Y$ defined by $P'(\pi(y))=P(y)$ be the canonical embedding.
Let $S = P' \circ \pi \circ T$. Then $\|S\| \leq \|\pi \circ T\|$. Also note that $\pi \circ S = \pi \circ T$. To see this note that for any $x \in X$, $S(x)-T(x) = P'(\pi(T(x))-T(x) = P(T(x))-T(x) \in Z$. Thus $\|S\|= \|\pi \circ T\|$ and $d(T,{\mathcal L}(X,Z))= d(S,{\mathcal L}(X,Z))\leq \|S\|= \|\pi\circ T\|$.
Therefore equality holds in the inequalities and we have the required conclusion.
	
\end{proof}
\begin{rem}
We have used the finite codimensionality to get $\pi \circ T$ compact. The theorem also holds by	assuming $\pi\circ T$ is compact. Also note that since $P$ is a contractive projection, $\|(x-P(x))-P(y)\|\geq \|P(y)\|$ for all $x,y \in X$, so in Theorem 6, the range of $P$ is contained in $Z^{\dagger}$.
As noted in Remark 2, if one considers these questions only in the class of compact operators, the minimax formulae from Proposition 1, are still valid. It was shown recently in \cite{R1} that these  can be used to compute the directional derivative of the operator norm at these operators.
\end{rem}


\begin{thebibliography}{99}
 		
 

 
 \bibitem{A} E. M. Alfsen, {\em Compact convex sets and boundary integrals}, Springer, New York, 1971.
 	
 	\bibitem{BS} R. Bhatia and P. Semrl, {\em Orthogonality of matrices and the distance problem}, Linear Algebra and its applications,
 287 (1999) 77--85.
 \bibitem{MBF} Monika, F. Botelho and R. J. Fleming,
 {\em The  Existence of linear selection and the quotient lifting property}, Indian J Pure Appl Math (2021). https://doi.org/10.1007/s13226-021-00171-z
 \bibitem{BFR} F. Botelho, R. J. Fleming and T. S. S. R. K. Rao, {\em Proximinality of subspaces and the Quotient Lifting Property}, Preprint 2022.
 
 \bibitem{DU} J. Diestel and J. J. Uhl, {\em Vector Measures}, Mathematical Surveys, 15, A. M. S 1977.		 		
 			
 \bibitem{HWW} P. Harmand, D. Werner and W. Werner, {\em $M$-ideals in Banach spaces and Banach algebras}, Springer LNM 1547, Berlin 1993.		 			
 		 		
	\bibitem{Lima}
 			A. Lima,
 			{\em  Intersection properties of balls in spaces of compact operators}, Ann. Inst. Fourier 28, 35--65 (1978).
 		
 		
 		\bibitem{PS} K. Paul and D.Sain, {\em Birkhoff-James orthogonality and its applications in the study of geometry of Banach spaces},
  Ruzhansky, Michael (ed.) et al., Advanced topics in mathematical analysis. Boca Raton, FL: CRC Press. 245--284 (2019),  Chapter 8,
 pp 245--285.
 \bibitem{R} T. S. S. R. K. Rao,{ \em Operators Birkhoff-James orthogonal to spaces of operators}, Numer. Funct. Anal. Optim. 42 (2021) 1201-1208.	
 	\bibitem{R1} T. S. S. R. K. Rao, {\em Subdifferential set of an operator}, Monatsh Math (2022). https://doi.org/10.1007/s00605-022-01739-5.	
 		
 \bibitem{R2} T. S. S. R. K. Rao, {\em Order preserving quotient lifting properties}, Positivity 26, 37 (2022). https://doi.org/10.1007/s11117-022-00907-z
 			
 \bibitem{Z} V. Zizler, {\em On some extremal problems in Banach spaces},
Math. Scand. 32(1973) 214--224.			
 			
		
 		

\end{thebibliography}
\end{document}